\documentclass[11pt,reqno,a4paper]{amsart}
\usepackage[toc,page]{appendix}

\usepackage[utf8]{inputenc}
\usepackage{esint}
\usepackage{MnSymbol}
\usepackage{enumitem}
\usepackage[english]{babel}
\usepackage{amsmath}
\usepackage{mathtools}
\usepackage{thm-restate}
\usepackage{comment}
\usepackage{bm}
\usepackage{tikz}
\usetikzlibrary{patterns}
\usetikzlibrary{arrows.meta}

\usepackage{xcolor}

\newtheorem{theorem}{Theorem}
\newtheorem*{theorem*}{Theorem}
\newtheorem{proposition}{Proposition}[section]
\newtheorem{lemma}[proposition]{Lemma}

\theoremstyle{definition}
\newtheorem{definition}[proposition]{Definition}

\numberwithin{equation}{section}

\usepackage[colorlinks,
pdfpagelabels,
pdfstartview = FitH,
bookmarksopen = true,
bookmarksnumbered = true,
linkcolor = blue,
plainpages = false,
hypertexnames = false,
citecolor = red] {hyperref}

\hypersetup{
linktoc=page
}

\usepackage[capitalize]{cleveref}
\crefname{enumi}{}{}
\crefname{equation}{}{}

\makeatletter
\def\@tocline#1#2#3#4#5#6#7{\relax
\ifnum#1>\c@tocdepth
\else
\par \addpenalty\@secpenalty\addvspace{#2}%
\begingroup \hyphenpenalty\@M{}
\@ifempty{#4}{%
\@tempdima\csname r@tocindent\number#1\endcsname\relax
}{%
\@tempdima#4\relax
}%
\parindent\z@ \leftskip#3\relax \advance\leftskip\@tempdima\relax
\rightskip\@pnumwidth{} plus4em \parfillskip-\@pnumwidth{}
#5\leavevmode\hskip-\@tempdima{}
\ifcase{#1}
\or\or{}\hskip 1em \or{}\hskip 2em \else \hskip 3em \fi%
#6\nobreak\relax
\dotfill\hbox{} to\@pnumwidth{\@tocpagenum{#7}}\par
\nobreak{}
\endgroup
\fi}
\makeatother

\def\R {\mathbb{R}}

\def\E {\mathbb{E}}

\def\osc{\operatorname{osc}}

\renewcommand{\tilde}{\widetilde}

\newcommand{\eps}{\varepsilon}

\newcommand{\pa}{\partial}

\renewcommand{\fint}{\strokedint}

\usepackage{etoolbox}
\newcommand{\capacity}[2]{
  \operatorname{cap}#1#2
}

\allowdisplaybreaks{}

\begin{document}

\title[A note on the capacity estimate in metastability]{A note on the capacity estimate in metastability for generic configurations}
\author[Avelin]{Benny Avelin}
\address{Benny Avelin,
Department of Mathematics,
Uppsala University,
S-751 06 Uppsala,
Sweden}
\email{\color{blue} benny.avelin@math.uu.se}
\author[Julin]{Vesa Julin}
\address{Vesa Julin,
Department of Mathematics and Statistics,
University of Jyv\"askyl\"a,
P.O. Box 35,
40014 Jyv\"askyl\"a,
Finland}
\email{\color{blue} vesa.julin@jyu.fi}


\date{\today}

\begin{abstract}
  In this paper we further develop the ideas from Geometric Function Theory initially introduced in~\cite{BeLaVe}, to derive capacity estimate in metastability for arbitrary configurations. The novelty of this paper is twofold. First, the graph theoretical connection enables us to exactly compute the pre-factor in the capacity. Second, we complete the method from~\cite{BeLaVe} by providing an upper bound using Geometric Function Theory together with Thompson's principle, avoiding explicit constructions of test functions.
\end{abstract}

\maketitle
\setcounter{tocdepth}{2}

\section{Introduction}

In this paper we continue the study of the capacity estimate from~\cite{BeLaVe}, where we introduce a geometric characterization of the Eyring-Kramers formula. To introduce our setting, we begin by considering the Kolmogorov process
\begin{align*}
  d X_t = -\nabla F(X_t) dt + \sqrt{2\eps} dB_t
\end{align*}
where $F$ is a non-convex potential  and $\eps$ is a small positive number. A formula for the expected transition time from one local minimum point to another was proposed independently by Eyring~\cite{Eyr} and Kramers~\cite{Kra} in the context of metastability of chemical processes, and can be stated as follows. Assume that $x$ and $y$ are quadratic local minima of $F$, separated by a unique saddle $z$ which is such that the Hessian has a single negative eigenvalue $\lambda_1(z)$. Then the expected transition time from $x$ to $y$ satisfies
\begin{align*}
  \E^x[\tau] \simeq \frac{2\pi}{|\lambda_1(z)|} \sqrt{\frac{|\det(\nabla^2 F(z))|}{\det(\nabla^2 F(x))}} e^{(F(z)-F(x))/\eps},
\end{align*}
where $\simeq$ denotes that the comparison constant tends to $1$ as $\eps \to 0$.
The validity of the above formula has been studied extensively, references can be found in for instance~\cite{BeLaVe,BGKBook,FW, LMS}. The first rigorous proof of the Eyring-Kramers formula above, is by~\cite{BGK} using  potential theory and this approach has turned out to be fruitful.

Our main motivation to study this phenomenon, comes from non-convex optimization, for instance, optimization of neural networks. In this setting, the minima/saddles are in general degenerate and/or non-smooth.

In~\cite{BeLaVe} we use the potential theoretic formulation and extend the results of~\cite{BGK,BerGentz} to more general cases, which in particular includes non-smooth critical points.
As in~\cite{BGK}, the main technical issue is to provide sharp capacity estimates and the main result in~\cite{BeLaVe} is a geometric characterization of Newtonian capacity w.r.t.~the measure $e^{-F(x)/\eps} dx$ inspired by the corresponding characterization for conformal capacity originally proved by Gehring,~\cite{G}. In~\cite{BeLaVe} we observe that the capacity depends on the  configuration of the saddle points which connect the two local minima, but we computed the capacity only in the simple cases when the saddles are either parallel or in series, see \cref{f:par:ser}. However, for an arbitrary smooth potential  the situation can be more complex and the configuration of the saddle points can be a combination of both parallel and series cases  with essentially arbitrary complexity.

\begin{figure}
  \begin{center}
    \begin{minipage}{\textwidth}
      \raisebox{-0.45\height}{
      \begin{tikzpicture}[scale=2]

        \filldraw (-1,0) circle (0.05);
        \draw[dashed] (-1,0) circle (0.2);
        \draw (-1,-0.3) node {$x_u$};
        \filldraw (1,0) circle (0.05);
        \draw[dashed] (1,0) circle (0.2);
        \draw (1,-0.3) node {$x_w$};

        \filldraw[color=gray!60] (0,1) circle (0.05);
        \draw (0,1.2) node {$z_1$};
        \filldraw[color=gray!60] (0,0.2) circle (0.05);
        \draw (0,0.4) node {$z_2$};
        \filldraw[color=gray!60] (0,-1) circle (0.05);
        \draw (0,-1+0.2) node {$z_3$};

        \draw[->,color=red] (-1,0) to[out=45,in=-180] (0,1) to[out=0,in=-180-45] (1,0);
        \draw[->,color=red] (-1,0) to[out=10,in=-180] (0,0.2) to[out=0,in=-180-10] (1,0);
        \draw[->,color=red] (-1,0) to[out=-45,in=-180] (0,-1) to[out=0,in=-180+45] (1,0);

        \draw[red,rotate around={0:(0,0)}] (-0.1,1+0.1) rectangle (0.1,1-0.1);
        \draw[red,rotate around={0:(0,0)}] (-0.1,0.2+0.1) rectangle (0.1,0.2-0.1);
        \draw[red,rotate around={0:(0,0)}] (-0.1,-1+0.1) rectangle (0.1,-1-0.1);

      \end{tikzpicture}}
      \raisebox{-0.5\height}{
      \begin{tikzpicture}[scale=2]

        \filldraw (-1,0) circle (0.05);
        \draw[dashed] (-1,0) circle (0.2);
        \draw (-1,-0.3) node {$x_u$};
        \filldraw (0.5,0.2) circle (0.05);
        \draw[dashed] (0.5,0.2) circle (0.2);
        \draw (0.5,0.2-0.3) node {$x_1$};
        \filldraw (2,0) circle (0.05);
        \draw[dashed] (2,0) circle (0.2);
        \draw (2,-0.3) node {$x_w$};

        \filldraw[color=gray!60] (-0.25,0.1) circle (0.05);
        \draw (-0.25,0.1+0.2) node {$z_1$};
        \filldraw[color=gray!60] (0.5+0.75,0.1) circle (0.05);
        \draw (0.5+0.75,0.1+0.2) node {$z_2$};
        \draw[->,color=red] (-1,0) to[out=10,in=-180+10] (-0.25,0.1) to[out=10,in=-180] (0.5,0.2) to[out=0,in=-180-10] (0.5+0.75,0.1) to[out=-10,in=-180-10] (2,0);

        \draw[red,rotate around={0:(0,0)}] (-0.25-0.1,0.1-0.1) rectangle (-0.25+0.1,0.1+0.1);
        \draw[red,rotate around={0:(0,0)}] (0.5+0.75-0.1,0.1-0.1) rectangle (0.5+0.75+0.1,0.1+0.1);
      \end{tikzpicture}}
    \end{minipage}
  \end{center}
  \caption{Left picture is the parallel case and the right is the series case, $x_u,x_w$ are local minimum points and $z_i$ are saddle points.}\label{f:par:ser}
\end{figure}
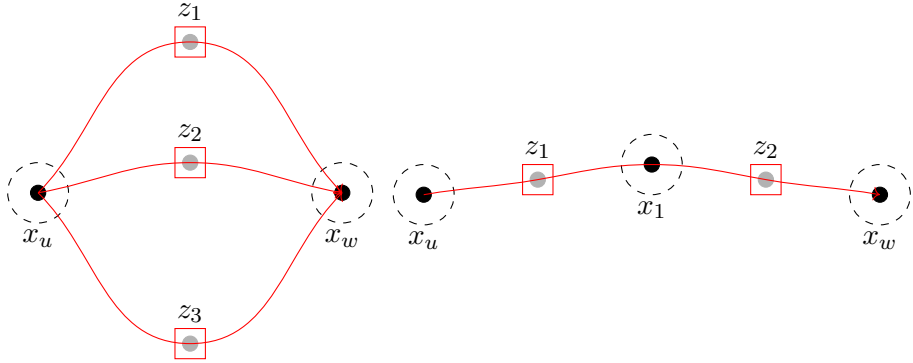

The novelty in~\cite{BeLaVe} was the use of Geometric Function Theory to provide a lower bound for the capacity. In this paper we complete this method by providing an upper bound using Geometric Function Theory together with Thompson's principle, see the proof of the upper bound in \cref{sec:mainthm}, Proof of \cref{mainthm}.
Our goal is to extend the capacity estimate from~\cite{BeLaVe} to the case of arbitrary configurations of critical points.
We do this by discretizing the problem where the `valleys'/`islands' around the local minimum points are the vertices and the regions around the saddle points which we call `bridges' are the edges. The local capacity of a bridge can be geometrically characterized using the results from~\cite{BeLaVe} and this defines the weights of the edges, thus turning the problem into a capacitary problem on a graph. Connecting problems of this type to graphs is similar to~\cite{Michel}, however, we did not find this particular problem in the literature. We note that the result in~\cite{BGK} covers only the case of the parallel configuration, see \cref{f:par:ser}.
Moreover, the framework of geometric function theory (see~\cite{BeLaVe}) makes this construction straightforward and natural.

The capacitary problem on the graph is equivalent to the notion of an electrical network, which was originally defined by Kirchhoff in the 1840s in his elegant solution to the problem of replacement resistance for a network of resistors~\cite{Kirchhoff}.
For a modern presentation of electrical networks and its connection to Markov chains and Kirchhoff's theorem, we refer to~\cite{Bollobas,Levin, Wagner}.

\subsection{Assumptions and definitions}\label{sec:ass}

In order to state our main results we first need to introduce our assumptions on the potential $F$. We remark that our assumptions cover the case where $F$ is a Morse function as defined in~\cite[Assumption 10.3]{BGKBook}, i.e.~a $C^2$ function in which all critical points are non-degenerate (non-degenerate Hessian with at most one negative eigenvalue). We further remark that our assumptions cover the degenerate case studied in~\cite{BerGentz}, but we also allow for non-smooth (Lipschitz) potentials.

Let us first introduce some general terminology. 
Recall that a Lipschitz function $h:\R \to \R$ has a \emph{critical point} at $t$, if $0$ is in the generalized gradient of $h$ at $t$ in the following sense,
\begin{align*}
  \limsup_{s \to t\pm} \frac{f(s)-f(t)}{s-t} \geq 0 
  \quad \text{or} \quad
  \liminf_{s \to t\pm} \frac{f(s)-f(t)}{s-t} \leq 0,
\end{align*}
where in the above we mean that both the left and the right limit satisfies the conditions.
We say that a point $z$ of a Lipschitz function $f:\R^n \to \R$ is a \emph{critical point}, if for every $e \in \R^n$, $\|e\|=1$, the function $h_e(t)=f(z+te)$ has a critical point at $0$.

Given a continuous function $f:\R^n \to \R$, we say that  a local minimum of $f$ at  $z$  is \emph{proper} if there exists a $\hat \delta > 0$ such that for every  $0 < \delta < \hat \delta$ there is a $\rho$ such that
\begin{equation*}
  f(x) \geq 
  \begin{cases}
    f(z), & x \in B_{\rho}(z), \\
    f(z) + \delta, & x \in \pa B_{\rho}(z),
  \end{cases}
\end{equation*}
where $B_\rho(z)$ denotes an open ball with radius $\rho$ centered at $z$ (proper maximum is defined analogously). When the center is at the origin we use the short notation $B_\rho$. We say that a critical point $z$ of $f$ is a \emph{saddle point} if it is not a proper local minimum nor maximum point.


Let us then proceed to our assumption on the potential.

\begin{definition}\label{def:admissible}
  Let $F \in C^{0,1}(\R^n)$ satisfy the following quadratic growth  condition
  \begin{equation*}
    F(x) \geq \frac{ |x|^2}{C_0}  - C_0
  \end{equation*}
  for a  constant $C_0 \geq 1$. We assume that every local minimum point $z$ of $F$ is proper.

  We say that $F$ is \emph{admissible} if for every saddle point $z \in \R^n$ of $F$ there are convex functions $g_z: \R \to \R$ and $G_z :\R^{n-1} \to \R$ which have proper minimum at $0$, such that $g_z(0) = G_z(0) = 0$, and  an isometry\footnote{Recall that a mapping $T$ is an isometry if $|T(x) - T(y)| = |x-y|$. In $\R^n$, this implies that $T(x) = Ax + b$, where $A$ is an orthogonal matrix. That is, $T$ consists of translation $b$ and a rotation $A$.} $T_z : \R^n \to \R^n$  such that, denoting $x = (x_1, x') \in \R \times \R^{n-1}$, it holds
  \begin{equation}\label{eq:struc3} 
    \big| (F\circ T_z) (x) -F(z)  + g_z(x_1) - G_z(x')\big|\leq \omega( g_z(x_1))  + \omega( G_z(x')),
  \end{equation}
  where $\omega : [0,\infty) \to [0,\infty)$ is a continuous and increasing function  with $\lim_{s \to 0} \frac{\omega(s)}{s} = 0$.
\end{definition}

The assumption \cref{eq:struc3} allows the saddle point to be degenerate, but we do not allow branching saddles, in the sense that $\{f(x) < f(z)\} \cap B_{\rho}(z)$ can have at most two components for small $\rho$.  Note that the convex functions $g_z, G_z$ and the isometry $T_z$ depend on $z$, while the function $\omega$ is the same for all saddle points. As such, we denote by  $\delta_0$ the largest number  for which $\omega(\delta) \leq \frac{\delta}{100}$ for all $\delta \leq 4 \delta_0$. 

\begin{definition}\label{def:bridges}
  Let $F \in C^{0,1}(\R^n)$ be admissible, then for every saddle point $z$ and $\delta > 0$, we define the \emph{bridge} at $z$ as
  \begin{equation*}
    O_{z,\delta} := T_z\left(  \{x_1 \in \R:  g_z(x_1) < \delta\} \times \{x' \in \R^{n-1}:  G_z(x') < \delta\}\right),
  \end{equation*}
  where $T_z$ is the isometry from \cref{def:admissible}. See \cref{f:localization}.
\end{definition}

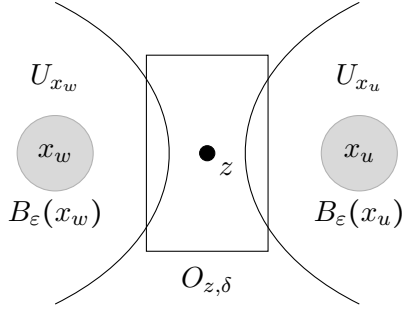
\begin{figure}
  \begin{center}
    \begin{tikzpicture}

      \draw (-2,-2) .. controls (0,-1) and (0,1) .. (-2,2);
      \draw (-2,1) node {$U_{x_w}$};
      \draw (2,-2) .. controls (0,-1) and (0,1) .. (2,2);
      \draw (2,1) node {$U_{x_u}$};

      \filldraw (0,0) circle (0.1);
      \draw (0.25,-0.2) node {$z$};

      \filldraw[color=gray!60, fill=gray!30] (-2,0) circle (0.5);
      \draw (-2,0) node {$x_w$};
      \draw (-2,-0.8) node {$B_\varepsilon(x_w)$};
      \filldraw[color=gray!60, fill=gray!30] (2,0) circle (0.5);
      \draw (2,0) node {$x_u$};
      \draw (2,-0.8) node {$B_\varepsilon(x_u)$};

      \draw (-0.8,-1.3) rectangle (0.8,1.3);
      \draw (0,-1.7) node {$O_{z, \delta}$};
    \end{tikzpicture}
  \end{center}
  \caption{The neighborhood $O_{z, \delta}$ of the saddle point $z$ (bridge) connects the sets $U_{x_u}$ and $U_{x_w}$, components of $\{F < F(z)-\delta/3\}$. }\label{f:localization}
\end{figure}

Note that, since the saddle may be flat, we should talk about sets rather than points. However, we adopt the convention that we always choose a representative point from each saddle (set) and thus we may label the saddles by points $z_1, z_2, \dots$. Moreover, we assume that there is a $\delta_1 \leq \delta_0$ such that for $\delta < \delta_1$ we have that if $z_1$ and $z_2$ are two different saddle points, then their neighborhoods $O_{z_1, 3\delta}$ and   $O_{z_2, 3\delta}$ defined in \cref{def:bridges} are disjoint.  
Furthermore, we also assume that $\eps_0$ is small enough that for any local minimum point $x$, the ball $B_{\eps_0}(x)$ does not intersect any bridge or any $\eps_0$ ball around any other local minimum.

We use the definitions of a geodesic length and a minimal cut originally defined in~\cite{BeLaVe}, inspired by~\cite{G}.
\begin{definition}\label{def:geodesic_cut}
  Let $A,B \subset \Omega \subset \R^n$ where $\Omega$ is a domain and $A \cap B = \emptyset$. We denote the \emph{curve family}
  \begin{align*}
    \mathcal{C}(A,B; \Omega) := \{\gamma: \gamma \in C^{1}([0,1];\Omega), \gamma(0) \in A, \gamma(1) \in B\}
  \end{align*}
  and  the family of \emph{separating sets} as $\mathcal{S}(A,B;\Omega)$, where  a smooth hypersurface $S \subset \R^n$ (possibly with boundary) is in $\mathcal{S}(A,B;\Omega)$  if  every
  $\gamma \in \mathcal{C}(A,B; \Omega)$  intersects $S$. We define the \emph{geodesic distance} between $A$ and $B$ in $\Omega$ as
  \begin{equation*}
    d_{\eps}(A,B; \Omega) := \inf\left( \int_{\gamma} |\gamma'| e^{\frac{F(\gamma)}{\eps}} \, dt  : \gamma \in \mathcal{C}(A,B; \Omega)  \right)
  \end{equation*}
  and   the \emph{minimal cut}  by
  \begin{equation*}
    V_{\eps}(A,B; \Omega) := \inf\left( \int_{S}  e^{-\frac{F(x)}{\eps}} \, d \mathcal{H}^{n-1}(x) :S \in  \mathcal{S}(A,B;\Omega) \right) .
  \end{equation*}
\end{definition}

We define some topological quantities.
\begin{definition}\label{def:islands}
  Let $x_u,x_w$ be two local minima of an admissible $F$. The \emph{communication height} between $x_u,x_w$ is defined as
  \begin{equation*}
    F(x_u; x_w) =  \inf_{\gamma \in \mathcal{C}(B_\eps(x_u),B_\eps(x_w); \R^n)} \sup_{t \in [0,1]}  \,  F(\gamma(t)).
  \end{equation*}
  Fixing $\delta < \delta_1$, we denote the component of the sub-levelset  $\{ F < F(x_u; x_w) + \delta/3\}$ which contains the points $x_u$ and $x_w$ by $U_{\delta/3}$, and we denote
  \begin{equation}
    \label{def:U-delta}
    U_{-\delta/3} := \{ F < F(x_u; x_w) - \delta/3\} \cap U_{\delta/3}.
  \end{equation}
  Furthermore, we remark that $F(x_u;x_w)$ does not depend on $\eps$ if $\eps < \eps_0$.
  We call the components of $U_{-\delta/3}$ \emph{islands}. { For each island $U$ we select a proper minimum point $x$ satisfying $F(x) = \min_U F$, and we will in the following denote $U_x$ as the island which contains $x$, see \cref{f:localization}}. We denote all saddle points in $U_{\delta/3} \setminus U_{-\delta/3}$ by $Z$.
\end{definition}

Finally we recall that the capacity of two disjoint  sets $A, B$ is defined as
\begin{align*}
  \capacity(A,B) = \inf \left(   \eps\int_{\R^n } |\nabla h|^2 e^{-\frac{F}{\eps}}\, dx \, : \,\, h=1 \text{ in } A, \, h \in W_0^{1,2}(\R^n \setminus B)\right).
\end{align*}

\subsection{Construction of the electrical network}\label{ss:construction}
\begin{definition}
  An \emph{electrical network} is a pair $(G,\bm y)$, where $G = (V,E)$ is a graph, where $V$ are the \emph{vertices} and $E$ are the \emph{edges}, the vector { $\bm y \in \R^{|E|}$} is called the \emph{admittances}.
\end{definition}

We will now construct an electrical network based on the islands and bridges from \cref{def:islands,def:bridges}. We associate the \emph{vertices} with the islands and for every vertex $v$ we denote the corresponding island by $U_v$.
The set of all vertices is $V$.
Furthermore, we associate the \emph{edges} with the bridges from \cref{def:bridges}, specifically, for every saddle point $z \in Z$ in \cref{def:islands} we associate the edge $e_z$ with the bridge { $O_{e_z} = O_{z, \delta}$}. The set of all edges is $E$,  vice versa we associate with $e \in E$ the corresponding saddle point $z_e \in Z$.
We say that vertices $v,v' \in V$ are \emph{incident} with an edge $e$, and vice versa,  if they are the ends of  the edge, or in other words,  the associated islands $U_v,U_v'$ intersect the  bridge $O_e$ (there are at most two since $F$ is admissible).  An edge which is incident with only one vertex is called a \emph{loop}. We also define a \emph{cycle} of the graph $G$ to be any non-trivial closed path for which only the first and last vertices are equal.

We thus have a graph $G = (V,E)$, and we orient it arbitrarily (i.e.~we orient each edge of $G$ arbitrarily by assigning an arrow on it pointing towards one of its two ends). In order to have an electrical network we need to define \emph{admittance} ${\bm y}_e$ for $e \in E$. Now, let $e \in E$, which is not a loop, and let $v_{-}, v_+ \in V$ be its incident vertices. Define the connected set $\Omega_{e} =  O_{z_e,\delta} \cup U_{v_-} \cup U_{v_+}$ and the admittance
\begin{equation}
  \label{eq:geometric}
  {\bm y}_e   :=   \eps \frac{V_{\eps}(B_\eps(x_{v_-}),B_\eps(x_{v_+}); \Omega_{e}) }{d_{\eps}(B_\eps(x_{v_-}),B_\eps(x_{v_+}); \Omega_{e}) } e^{\frac{F(z_e)}{\eps}}.
\end{equation}

From the geometric characterization of capacity in~\cite{BeLaVe}, we see that the admittance of the edge $e$ is the pre-factor of the capacity of $(B_{\eps}(x_{v_-}),B_{\eps}(x_{v_+}))$ in $\Omega_{e}$.
If $e$ is a loop we set ${\bm y}_e = 0$. We have thus constructed our  electrical network $(G,\bm y)$ which consists of the graph $G$ and the admittance vector $\bm y \in \R^{|E|}$.

\subsection{The Main Result}\label{ss:main}

We begin with some notation and recalling some results from~\cite{BeLaVe}.

For functions $f$ and $g$, which depend  continuously on $\eps>0$, we adopt the notation
\begin{equation*}
  f(\eps) \simeq g(\eps)
\end{equation*}
when there exists a constant $C$ depending only on the data of the problem such that
\begin{equation*}
  (1- C \hat \eta(\eps)) f(\eps) \leq g(\eps) \leq (1+C \hat \eta(\eps)) f(\eps),
\end{equation*}
where $\hat \eta(\cdot)$ is an increasing and continuous function $\hat \eta(\cdot):[0,\infty) \to [0,\infty)$ with $\lim_{s \to 0} \hat \eta(\cdot) = 0$. 
We remark that in the following the function $\hat \eta$ is the one from \cref{lem:geometric-d} and \cref{lem:convex}. For us the explicit form will not be important but can be found in~\cite{BeLaVe} and we merely note that $\hat \eta$ is sublinear and depends on the Lipschitz constant of $F$ inside $U_{\delta}$, the dimension and the function $\omega$ from \cref{eq:struc3}. 

We need the above notation in order to relate the geodesic distance and the minimal cut from \cref{def:geodesic_cut} to the convex functions $g_z,G_z$ from \cref{def:bridges}, which is stated in the following proposition (for the proof see~\cite[Proposition 4.1--4.2]{BeLaVe}):

\begin{proposition}\label{lem:geometric-d}
  Let $v_-,v_+ \in V$ be incident vertices connected with edge $e \in E$ and let $0 < \eps < \eps_0$. Denote $x_{v_-},x_{v_+}$ the corresponding proper local minimum points and let $z_e$ be the corresponding saddle, then
  \begin{equation*}
    d_{\eps}(B_\eps(x_{v_-}),B_\eps(x_{v_+}); \Omega_e) \simeq e^{\frac{F(z_e)}{\eps}}  \int_{\R} e^{-\frac{g_{z_e}(y_1)}{\eps}} \, dy_1 ,
  \end{equation*}
 and
  \begin{equation*}
    V_{\eps}(B_\eps(x_v),B_\eps(x_{v'});\Omega_e) \simeq e^{-\frac{F(z_e)}{\eps}}  \int_{\R^{n-1}} e^{-\frac{G_{z_e}(y')}{\eps}} \, dy' ,
  \end{equation*}
  where $g_{z_e},G_{z_e}$ are the functions in \cref{def:bridges}.
\end{proposition}

We need the definition of a spanning tree for Kirchhoff's formula.
\begin{definition}
  Let $G=(V,E)$ be a graph. We say that $G'$ is a \emph{spanning subgraph} of $G$ if $V(G') = V(G)$ and $E(G') \subset E(G)$ (i.e.~the same vertices but only a subset of the edges). A \emph{tree} is a
  connected graph which does not contain cycles and \emph{a spanning tree} of $G$ is a spanning subgraph of $G$ that is a tree. We denote the set of all spanning trees of $G$ by $\mathcal{T}(G)$. Finally, for two vertices $v,w \in V$ we let $G/vw$ denote the graph obtained by merging the vertices $v$ and $w$ together into a single vertex.
\end{definition}

We are now ready to state our main theorem.

\begin{theorem}\label{mainthm}
  Let $F$ be admissible as in \cref{def:admissible}, let $x_u$ and $x_w$ be local minimum points of $F$ and let $(G,\bm y)$ be the electrical network as in \cref{ss:construction}. Let $u,w$ be the associated vertices in $V$. Then the capacity is given by
  \begin{align*}
    \capacity(B_\eps(x_u),B_\eps(x_w)) \simeq \frac{T(G;{\bm y})}{T(G/uw;{\bm y})},
  \end{align*}
  where
  \begin{align}\label{def:T-thing}
    T(G;{\bm y}) = \sum_{G' \in \mathcal{T}(G)} \Big(  \prod_{e \in G'} {\bm y}_e \Big).
  \end{align}
\end{theorem}

\cref{mainthm}, together with  the formula \eqref{eq:geometric}, provide the characterization of the capacity in the general case where the critical points may have any configuration.

\section{Preliminaries on graph theory and electrical networks}
In this section we recall some basic results in graph  theory. For an introduction to the topic  we refer to~\cite{Bollobas, Levin, Wagner}.

The signed incidence matrix  $D$ of the oriented graph $G=(V,E)$  is the $|V| \times |E|$ matrix with entries
\begin{align*}
  D_{ve} =
  \begin{cases}
    +1 & \text{if $e$ points into $v$ but not out} \\
    -1 & \text{if $e$ points out of $v$ but not in} \\
    0 & \text{otherwise}.
  \end{cases}
\end{align*}
Let $\bm y$ be a vector of admittances defined in \cref{ss:construction}. Let $Y$ be the $|E|\times|E|$ diagonal matrix that has $\bm y$ as its entries, i.e., $Y = \text{diag}({\bm y}_e: e \in E)$. We also define the  weighted Laplacian matrix as $L = DYD^T$.

We begin by recalling the weighted Matrix-Tree theorem, see~\cite[Theorem 5]{Wagner}, which relates the quantity \cref{def:T-thing} to the weighted Laplacian matrix.
\begin{proposition}\label{lem:matrix:tree}
  Let $G = (V,E)$ be an oriented graph and let $D, Y$ and $L$  be as above. Then, for any $v \in V$
  \begin{align*}
    T(G;{\bm y}) = \det L(v \mid v)
  \end{align*}
  where $L(v \mid v)$ is $L$ with the row and column corresponding to $v$ removed and $T(G;{\bm y}) $ is defined in~\cref{def:T-thing}.
\end{proposition}

Let us recall Kirchhoff's theorem, see~\cite[Theorem 8]{Wagner}, { which relates the right-hand side of the formula in \cref{mainthm} to the solution of a linear system}.

\begin{proposition}\label{kirchhoff}
  {Let $G = (V,E)$ be an oriented graph, $(G,\bm y)$ the electrical network, and let $L$ be the corresponding weighted Laplacian matrix.}  Fix $u,w \in V$ and let the vector  ${\bm \varphi}  \in \R^{|V|}$, with the component ${\bm \varphi}_u = 0$, be the solution  to the system
  \begin{align*}
    L {\bm \varphi} = {\bm \delta}_w,
  \end{align*}
  where $\bm \delta_w$ is a vector with $1$ in the position of $w$ and $0$ otherwise. Then  the component ${\bm \varphi}_w$
  is given by
  \begin{align*}
    {\bm \varphi}_w = \frac{T(G/uw;{\bm y})}{T(G;{\bm y})}.
  \end{align*}
\end{proposition}

The classical interpretation of Kirchhoff's theorem is that of a network of resistors (the admittance is the inverse of the resistance), where we have grounded one end of the network ($\bm \varphi_u = 0$) and let $1$ Ampere of current flow through it (right-hand side $\bm \delta_w$). Then the voltage at the exiting node $\bm \varphi_w$ is given by the formula above. This allowed Kirchhoff (\cite{Kirchhoff}) to solve the problem of replacement resistance which in this case is just $\bm \varphi_w$.

Given an electrical network $(G,\bm y)$ we may define a discrete Dirichlet capacity between two vertices $v_1,v_m \in V$ as
\begin{align*}
  \min_{{\bm \varphi} \in \R^{m}; {\bm \varphi}_1 = 1; {\bm \varphi}_m = 0} \langle L {\bm \varphi}, {\bm \varphi} \rangle
\end{align*}
where $L$ is the weighted Laplacian matrix. Then the minimizer of the above problem is inversely related to Kirchhoff's theorem, \cref{kirchhoff}. For more information, see~\cite{Bollobas}.

\begin{lemma}\label{bennyLemma}
  Let $(G,\bm y)$ be the electrical network from \cref{ss:construction} and let $L$ be the Laplacian matrix. Then it holds
  \begin{align*}
    \min_{{\bm \varphi} \in \R^{m}; {\bm \varphi}_1 = 1; {\bm \varphi}_m = 0} \langle L {\bm \varphi}, {\bm \varphi} \rangle = \frac{T(G;{\bm y})}{T(G/v_1v_m;{\bm y})}.
  \end{align*}
  The minimizer is given by the unique solution with the boundary conditions $\bm \varphi_m=0$, $\bm \varphi_1=1$ to the linear system
  \begin{align*}
    L {\bm \varphi} = \lambda ( {\bm \delta}_1 - {\bm \delta}_m)
  \end{align*}
  where ${\bm \delta}_1 = (1,0,\ldots,0)$ and ${\bm \delta}_m = (0,\ldots,0,1)$ are  vectors of length $m$ and $\lambda$ is the value of the minimum problem.
\end{lemma}
\begin{proof}
  Recall that $L = D Y D^T$, where $D$ is the signed incidence matrix and $Y$ is the admittance matrix. Let us first reduce the problem. 
  Note that the constraint ${\bm \varphi}_m = 0$ implies that we may remove the last row of $D$ (call it $D_-$) and the last entry of ${\bm \varphi}$ (call it ${\bm \varphi}_-$) and note that $D_-^T {\bm \varphi}_- = D^T {\bm \varphi}$. Let $L_- = D_- Y D_-^T = L(v_m \mid v_m)$ and note that similar reasoning gives that
  \begin{align*}
    \langle L_- {\bm \varphi}_-, {\bm \varphi}_- \rangle = \langle L {\bm \varphi},{\bm \varphi} \rangle.
  \end{align*}
  By the Lagrange multiplier method we get
  \begin{align*}
    \begin{cases}
      L_- {\bm \varphi}_- &= \lambda {\bm \delta_1} \\
      ({\bm \varphi_-})_1 &= 1,
    \end{cases}
  \end{align*}
  where ${\bm \delta_1} = (1,0,\ldots)$. Note that by \cref{lem:matrix:tree}
  we know that $\det(L_-) = T(G;{\bm y}) \neq 0$ which gives that the above system has a unique solution. From the above we get that the value of the minimum is given as
  \begin{align}\label{eq:shit1}
    \langle L {\bm \varphi}, {\bm \varphi} \rangle = \langle L_- {\bm \varphi}_-, {\bm \varphi}_- \rangle = \lambda.
  \end{align}
  Next, we note that $\bm \varphi / \lambda$ is a solution to the linear system in \cref{kirchhoff}, as such we get 
  \begin{align*}
    \frac{T(G/uw;{\bm y})}{T(G;{\bm y})} = \frac{\bm \varphi_1}{\lambda} = \frac{1}{\lambda},
  \end{align*}
  which together with \cref{eq:shit1} finishes the proof.
\end{proof}

We also need the following dual formulation of the minimization problem in \cref{bennyLemma}.

\begin{lemma}\label{lem:dual-net}
  Let $G = (V,E)$ be an oriented graph, where $V = (v_1,\ldots,v_m)$, let $D$ be the signed incidence matrix, $L$ the Laplacian matrix, and let $Y$ be the admittance matrix. Then it holds
  \begin{align}\label{eq:shit2}
    \min \left( \langle Y^{-1} {\bm j} , {\bm j} \rangle : \,\,  {\bm j} \in \R^{|E|},  \,\, D {\bm j} = {\bm \delta}_1 -  {\bm \delta}_m  \right) = \frac{1}{\lambda},
  \end{align}
  where $\lambda$ is the value of the minimization problem from \cref{bennyLemma}, i.e.,
  \begin{align*}
    \lambda =  \min_{{\bm \varphi} \in \R^{m}; {\bm \varphi}_1 = 1; {\bm \varphi}_m = 0} \langle L {\bm \varphi}, {\bm \varphi} \rangle.
  \end{align*}
\end{lemma}

We point out that one may interpret the minimization problem \cref{eq:shit2} as a discrete version of Thompson's principle.

\begin{proof}
  Let ${\bm j}\in \R^{|E|}$ be the minimizer of~\cref{eq:shit2}. The first variation of the minimization problem implies that $\langle Y^{-1} {\bm j},\bm e \rangle = 0$ for all ${ \bm e} \in \R^{|E|}$ with $D{ \bm e} = 0$, i.e.,
  \begin{equation} \label{eq:thompson-1}
    Y^{-1} {\bm j} \in \text{Ker}^{\perp}(D).   
  \end{equation}
  Recall that the solution to the minimization problem in~\cref{bennyLemma} satisfies
  \begin{align*}
    \bm \delta_1 - \bm \delta_m = \lambda^{-1} L \bm \varphi = \lambda^{-1} DYD^T \bm \varphi = D(\lambda^{-1} Y D^T \bm \varphi)
  \end{align*}
  as such $\tilde {\bm j} = \lambda^{-1} YD^T \bm \varphi$ will according to the above satisfy the constraint $D \tilde{ \bm j} = {\bm \delta}_1- {\bm \delta}_m$. 
  Then, $\tilde {\bm j} - \bm j =: \alpha \in \text{Ker}(D)$ and it holds trivially $\langle D^T{\bm \varphi}, \alpha \rangle = 0$ since, $D^T{\bm \varphi} \in \text{Ker}^{\perp}(D)$. Moreover, by~\cref{eq:thompson-1} it holds $\langle Y^{-1} {\bm j}, \alpha \rangle = 0$, thus
  \begin{align*}
    \langle Y^{-1}\alpha, \alpha \rangle = \langle  \lambda^{-1} D^T{\bm \varphi}, \alpha \rangle - \langle Y^{-1}{\bm j}, \alpha \rangle = 0.
  \end{align*}
  Since $Y^{-1}$ is positive definite we obtain $\alpha = 0$, that is
  \begin{equation*}
    \lambda^{-1} Y D^T{\bm \varphi} = {\bm j}.
  \end{equation*}
  The result then follows from~\cref{bennyLemma} as 
  \begin{equation*}
    \langle Y^{-1} {\bm j} , {\bm j} \rangle = \frac{\langle Y^{-1} Y D^T{\bm \varphi} , Y D^T{\bm \varphi}  \rangle}{\lambda^2}   =  \frac{\langle D Y D^T{\bm \varphi},{\bm \varphi}  \rangle}{\lambda^2} = \frac{1}{\lambda}.
  \end{equation*} 
\end{proof}

\subsection{Simplification of the electrical network}

The formula in the statement of \cref{mainthm} given by Kirchhoff's formula is precise, but if the graph contains many cycles and loops, it may be unnecessarily cumbersome to calculate. In the next two lemmas we consider the case when the formula in \cref{mainthm}  can be simplified.

Consider a graph $G = (V,E)$. A \emph{cut vertex} is a vertex, that when removed from $G$ will increase the number of components. A \emph{biconnected} graph is a graph with no cut vertices. A \emph{biconnected component} of a graph $G$ is a maximal biconnected subgraph.

\begin{lemma}\label{lem:biconnect}
  Let $G=(V,E)$ be a graph with a biconnected component $G_1=(V_1,E_1)$ and let $G_2=(V_2,E_2)$ be a subgraph of $G$ such that they intersect in one cut vertex $v \in V$ and $G = G_1 \cup G_2$. Then if $\mathbf{y} \in \R^{|E|}$ is the admittance vector, $\mathbf{y_1} = \mathbf{y}\lvert_{E_1}$ and $\mathbf{y_2} = \mathbf{y}\lvert_{E_2}$, it holds
  \begin{align*}
    T(G;\mathbf{y}) = T(G_1;\mathbf{y_1})T(G_2;\mathbf{y_2}).
  \end{align*}
\end{lemma}
\begin{proof}
  By the definition of biconnected components, and since $G_1,G_2$ intersect only in $v$, we can by reordering the vertices write the Laplacian matrix $L = D Y D^T$ such that the first rows/columns correspond to the vertices in $G_1$. Then $L$ with the column and row corresponding to $v$ removed ($L(v \mid v)$) has a block diagonal structure with the blocks $L_1 = L_{G_1}(v \mid v)$ and $L_2 = L_{G_2}(v \mid v)$. Now, since $\det(L) = \det(L_1)\det(L_2)$ the claim follows by applying \cref{lem:matrix:tree} on all matrices.
\end{proof}

We can use the above lemma to simplify the computation of Kirchhoff's theorem in the presence of irrelevant biconnected components, see \cref{pic:biconnected}.

\begin{proposition}\label{lem:biconnect:kirchhoff}
  Consider the graph $G = (V,E)$ and let $Y$ be the admittance matrix. Assume that $G =  G_1 \cup G_2$, where $G_1$ is a biconnected component and $G_1, G_2$ intersect in a cut vertex $v \in V$. Then if $u,w \in V_2$, it holds
  \begin{align*}
    \frac{T(G;{\bm y})}{T(G/uw;{\bm y})} = \frac{T(G_2;{\bf  y}_2)}{T(G_2/uw;{\bm y}_2)}.
  \end{align*}
\end{proposition}

The main consequence of \cref{lem:biconnect:kirchhoff} is that, using the terminology from~\cite{BeLaVe}, only the vertices in $V_2$ are relevant. We also point out that this is related to the definition of a gate in~\cite{BGK}. In particular, referring to Kirchhoff's theorem, a consequence of the above is that the voltage $\bm \varphi$ is constant on the biconnected components and is thus redundant.

\begin{figure}
  \begin{tikzpicture}[scale=0.7]
    \begin{scope}[every node/.style={circle,thick,draw}]
      \node (A) at (-1,0) {a};
      \node (B) at (2.5,1) {};
      \node (C) at (2.5,-1) {};
      \node (D) at (6,0) {b} ;

      \node (F) at (1,-2) {};
      \node (G) at (4,-2) {};
    \end{scope}

    \begin{scope}[
      every node/.style={fill=white,circle},
      every edge/.style={draw=red,very thick}]
      \path [->] (A) edge (B);
      \path [->] (B) edge (D);
      \path [->] (A) edge (C);
      \path [->] (C) edge (D);
      \path [->] (C) edge (B);
    \end{scope}
    \begin{scope}[
      every node/.style={fill=white,circle},
      every edge/.style={draw=blue,very thick}]
      \path [->] (C) edge (F);
      \path [->] (F) edge (G);
      \path [->] (G) edge (C);
    \end{scope}
  \end{tikzpicture}
  \caption{Example of the graph decomposition in \cref{lem:biconnect:kirchhoff}. Here the subgraph corresponding to the blue edges is the biconnected component and the red edges correspond to the graph $G_1$.}\label{pic:biconnected}
\end{figure}
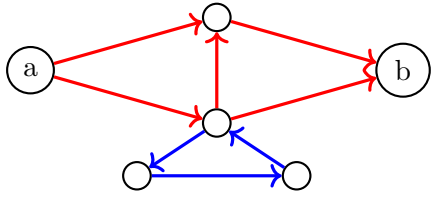

A consequence of \cref{bennyLemma} is that edges with small admittance does not contribute total capacity unless they significantly alter the topology of the graph:

\begin{lemma}[Deletion of edge]
  Let $(G,\bm y)$ be the electrical network as in \cref{bennyLemma}. Let $e \in E$ and define $G' = (V, E \setminus \{e\})$, then it holds
  \begin{align*}
    \frac{T(G';\bm y)}{T(G'/(v_1v_m);\bm y)} \leq \frac{T(G;\bm y)}{T(G/(v_1v_m);\bm y)} \leq \frac{T(G';\bm y)}{T(G'/(v_1v_m);\bm y)} + {\bm y}_e
  \end{align*}
\end{lemma}
\begin{proof}
  Let $Y'$ be the diagonal matrix $Y$ with the entry corresponding to ${\bm y}_e$ replaced by $0$.
  Then we immediately have
  \begin{align*}
    \min_{{\bm \varphi} \in \R^{m}; {\bm \varphi}_1 = 1; {\bm \varphi}_m = 0} \langle D Y' D^T {\bm \varphi}, {\bm \varphi} \rangle \leq \min_{{\bm \varphi} \in \R^{m}; {\bm \varphi}_1 = 1; {\bm \varphi}_m = 0} \langle D Y D^T {\bm \varphi}, {\bm \varphi} \rangle
  \end{align*}
  which proves the first inequality. For the second, note that for any edge $e \in E$, let $v_-,v_+ \in V$ be the incident vertices, then $|\varphi(v_-)-\varphi(v_+)| \leq 1$, hence for any $\bm y$ having each component bounded by $1$ satisfies
  \begin{align*}
    \langle D Y D^T {\bm \varphi}, {\bm \varphi} \rangle \leq  \langle D Y' D^T {\bm \varphi}, {\bm \varphi} \rangle + {\bm y}_e
  \end{align*}
  which proves the last inequality.
\end{proof}

\section{Proof of the main theorem}

The proof of the main theorem consists of an upper and a lower bound of the capacity. The lower bound uses the electrical network defined in \cref{ss:construction} and the variational definition of capacity, similar to the proof in~\cite{BeLaVe}. For simple networks the lower bound follows from the variational characterization of capacity together with the fundamental theorem of calculus.

For the upper bound, we provide a novel proof using ideas from Geometric Function Theory together with Thompson's principle which is in a sense dual to the lower bound. In this case, for simple networks the upper bound follows from Thompson's principle together with the divergence theorem. For the general case we need an alternative construction of the electrical network.

\subsection{Alternative construction of the electrical network}\label{sebsec:construction}

We will construct the alternative electrical network using the domain $U_{\delta/3}$ (see \cref{def:islands}), instead of using the components of $U_{-\delta/3}$ as in \cref{ss:construction}.
To this aim, for a saddle point $z \in Z$, we define the surface
\begin{equation}\label{def:surface}
  S_z  := T_z\left(  \{0 \} \times \{x' \in \R^{n-1}:  G_z(x') < \delta\}\right),
\end{equation}
where $T_z$ is from \cref{def:admissible}.
The set $U_{\delta/3}$ is connected, but the surfaces $S_z$ in~\cref{def:surface}  divide it into different components, which we will associate with vertices, see \cref{f:localization3}. Define
\begin{equation}
  \label{def:omega-delta}
  \Omega_{\delta/3} : = U_{\delta/3}\setminus \bigcup_{z \in Z} S_z.
\end{equation}

We will now provide two technical lemmas. The first says that any path connecting two local minimum points in $U_{-\delta/3}$ necessarily passes through a surface $S_z$ for some $z$ in the set of saddles $Z$, where we recall that $Z$ denotes the saddle points inside $U_{\delta/3} \setminus U_{-\delta/3}$. The second lemma states that $U_{-\delta/3}$ and $ \Omega_{\delta/3}$ have the same number of components and  $ \Omega_{\delta/3}$ defines exactly the same graph $G=(V,E)$ as in \cref{ss:construction}. 

\begin{lemma}\label{vesaTopology1}
  Let $U_{v}$ and $U_{v'}$ be two different components of $U_{-\delta/3}$ and let $\gamma \in  \mathcal{C}(U_{v},U_{v'}; U_{\delta/3})$. Then there is a critical point $z \in Z$ such that the intersection $\gamma([0,1]) \cap S_z$ is non-empty.
\end{lemma}

\begin{proof}
  W.L.O.G.~we assume $F(x_u;x_w)= 0$. Fix $\gamma_0 \in  \mathcal{C}(U_{v},U_{v'}; U_{\delta/3})$ and denote $\gamma \sim \gamma_0$ when $\gamma$ is homotopy equivalent to $\gamma_0$ in $U_{\delta/3}$. Define
  \begin{equation*}
    F_{\gamma_0}:=  \inf_{\gamma  \sim \gamma_0} \sup_{t \in [0,1]}  \,  F(\gamma(t)).
  \end{equation*}
  Then there is a critical point $z$ of $F$ such that $F(z) = F_{\gamma_0}$ and a continuous path $\gamma_1 \sim \gamma_0$ such that $\gamma_1(t) = z$ for some $t \in (0,1)$.
  We may choose the coordinates in $\R^n$ such that $z = 0$ and $S_z = S_0 =  \{0 \} \times \{x' \in \R^{n-1}:  G(x') < \delta\}$.

  Note that $S_0$ is a convex hypersurface with boundary $\partial S_0 =  \{0 \} \times \{x' \in \R^{n-1}:  G(x') = \delta\}$, and note that $\partial S_0$ is homeomorphic to $\mathbb{S}^{n-2}$.
  Since $F$ is admissible it follows from \cref{eq:struc3} that $F(x)\geq  F(0) + 2\delta/3$ on $x \in \partial S_0$ and therefore  since $F(0) > -\delta/3$ we have $\partial S_0 \subset \R^n \setminus U_{\delta/3}$. In particular,  if  $\gamma $  is a path in $U_{\delta/3}$    then it does not intersect $\partial S_0$, and  if $\gamma \sim \gamma_0 $ then $\gamma$ has to intersect $S_0$.  The claim then follows from  $\gamma_1 \sim \gamma_0$.
\end{proof}

\begin{lemma}\label{vesaTopology2}
  The set $\Omega_{\delta/3} $ defined in~\cref{def:omega-delta} has the same components as $ U_{-\delta/3}$ defined in~\cref{def:U-delta}. To be more precise, if $\Omega'$ is a component of $\Omega_{\delta/3} $ then there is exactly one component, say $U'$, of $U_{-\delta/3}$ such that $U' \subset \Omega'$.
\end{lemma}

\begin{proof}
  W.L.O.G.~we assume $F(x_u;x_w)= 0$. Let us fix a component $\Omega'$ of $\Omega_{\delta/3} $.
  Since $F$ is admissible, then for any $z \in Z$, we see from the definition of $S_z$ in~\cref{def:surface} that $F(x) \geq F(z)$ for all  $x \in S_z$, and hence $S_z \cap U_{-\delta/3} = \emptyset$. Thus, there is a component $U'$ of $U_{-\delta/3}$ such that $U' \subset \Omega'$.
  Let us also note that $U'$ is the only component of $U_{-\delta/3}$ which is in $\Omega'$, since if there was another component $U''$ then a curve $\gamma \in \mathcal{C}(U',U''; \Omega') \subset \mathcal{C}(U',U''; U_{\delta/3})$ necessarily intersects one $S_z$ by \cref{vesaTopology1}.
\end{proof}

We will localize the capacity of the sets $A = B_\eps(x_u)$ and $B = B_\eps(x_w)$ in $U_{\delta/3}$ by defining
\begin{equation}\label{def:capa-local}
  \capacity(A,B; U_{\delta/3}) := \inf \left( \eps \int_{U_{\delta/3}} |\nabla h|^2 e^{-\frac{F}{\eps}}\, dx: h=1 \text{ in } A,  h \in W_0^{1,2}(\R^n \setminus B) \right).
\end{equation}
In the above minimization problem we do not have any boundary condition on $\partial U_{\delta/3}$. Thus, it follows from a classical result of calculus of variations (see~\cite[Sec 2.4]{GiaHild}) that the minimizer $\hat h_{A,B}$ of~\cref{def:capa-local} satisfies the natural boundary condition, $\nabla \hat h_{A,B} \cdot n = 0$ on the smooth part of $\partial U_{\delta/3}$.

It is easy to see that for the localization~\cref{def:capa-local} it holds
\begin{equation}
  \label{eq:capa-local}
  \capacity(A,B) \geq   \capacity(A,B; U_{\delta/3}) \geq   (1-C \hat \eta(\eps))\capacity(A,B),
\end{equation}
where $\hat \eta$ is as in \cref{ss:main}.
Indeed, the first inequality in~\cref{eq:capa-local} is trivial. For the second we take $\hat h_{A,B}$ to be the minimizer of~\cref{def:capa-local} and we recall the rough capacity bound from~\cite[Lemma 3.2]{BeLaVe}, i.e., there exists constants $c_1,c_2,q_1,q_2$ such that
\begin{align}\label{eq:rough_cap}
  c_1 \eps^{q_1} e^{-F(x_u;x_w)/\eps} \leq \capacity(A,B) \leq c_2 \eps^{q_2} e^{-F(x_u;x_w)/\eps}.
\end{align}
We choose a cut-off function $0 \leq \zeta \leq 1$ such that $\zeta = 1$ in $U_{\delta/6}$, $\zeta = 0$ outside $U_{\delta/3}$   and $|\nabla \zeta | \leq C$, where $C$ depends on $\delta$ and on the Lipschitz constant of the potential $F$. Then, using Young's inequality, the maximum principle and \cref{eq:rough_cap} we get
\begin{align*}
  \begin{split}
    \capacity(A,B; U_{\delta/3}) &\geq \eps \int_{U_{\delta/3}} |\nabla  h_{A,B}|^2  \zeta^2 e^{-\frac{F}{\eps}}\, dx\\
    &\geq  \frac{\eps}{1+ \eps} \int_{U_{\delta/3}}|\nabla  (\hat h_{A,B} \zeta) |^2  e^{-\frac{F}{\eps}}\, dx -  \frac{2}{\eps} \int_{U_{\delta/3}} |\nabla \zeta|^2 \hat h_{A,B}^2 e^{-\frac{F}{\eps}}\, dx\\
    &\geq  \eps(1-2\eps) \int_{\R^n} |\nabla  (\hat h_{A,B} \zeta) |^2   e^{-\frac{F}{\eps}}\, dx -  \frac{C}{\eps} e^{-\frac{\delta}{6 \eps}} e^{-F(x_u,x_w)/\eps} \\
    &\geq (1 -  C \hat \eta(\eps))  \capacity(A,B),
  \end{split}
\end{align*}
where the last inequality follows from the sub-linearity of $\hat \eta$.

\subsection{Thompson's principle} 
The construction of the network via~\cref{def:omega-delta} is  suitable for the dual definition of the capacity via Thompson's principle. This is done by defining a class of vector fields, denoted by $\mathcal{M}$, where $X \in \mathcal{M}$  if  $X \in W^{1,\infty}(U_{\delta/3} \setminus (\bar A \cup \bar B); \R^n)$ and satisfies
\begin{align}\label{def:def-M} 
  \begin{cases}
    \text{div} X = 0 & \text{in }U_{\delta/3} \setminus (\bar A \cup \bar B),  
    \\
    X \cdot n = 0 & \text{on }\partial U_{\delta/3} 
    \\
    \int_{\partial A} X \cdot n = 1.
  \end{cases}
\end{align}
We note that the set $\mathcal{M}$ is non-empty, since the vector field $X = C e^{-V/\eps} \nabla \hat h_{A,B}$, where $C = (\capacity(A,B; U_{\delta/3}))^{-1}$, belongs to $\mathcal{M}$.
Then we have the following (see e.g.~\cite{LMS})
\begin{equation}\label{def:thompson}
  \frac{1}{\capacity(A,B;U_{\delta/3})}
  = 
  \inf \left( \eps \int_{U_{\delta/3}  \setminus (\bar A \cup \bar B)} |X|^2 e^{\frac{F}{\eps}}\, dx \, : \, X  \in \mathcal{M} \right).
\end{equation}

Let $G=(V,E)$  be the graph constructed as above using the domain $\Omega_{\delta/3}$ defined in~\cref{def:omega-delta} and let $X \in \mathcal{M}$. We construct a current $j :E \to \R $ associated with $X$ as follows.  Let us fix a vertex $v \in V \setminus \{ u,w\}$ and let $\tilde U_v$ be the associated component of the domain $\Omega_{\delta/3} $.  Denote the edges incident with $v$  by $e \in E_v \subset E$ and  the associated surface defined in~\cref{def:surface} by $S_e = S_{z_e}$.  The boundary  $\partial \tilde U_v$ is contained in $\partial U_{\delta/3} \cup (\bigcup_{e \in E_v} S_e)$.
Recall that $v \neq u,w$, therefore  $\text{div} X = 0$ in $\tilde U_v$, and we have by the divergence theorem and by $X \cdot n = 0 $ on  $\partial U_{\delta/3}$ that
\begin{equation}\label{eq:div-thm}
  0 =    -\int_{\tilde U_v} \text{div} (X) \, dx = \int_{\partial \tilde U_v} X \cdot n \, d \mathcal{H}^{n-1} = \sum_{e \in E_v} \int_{S_e} X \cdot n \, d \mathcal{H}^{n-1}.
\end{equation}
We define the value of $j$ at $e \in E_v$ as
\begin{equation}\label{def:current}
  j(e)  := \begin{cases} &\eps \int_{S_e} X \cdot n \, d \mathcal{H}^{n-1}, \,\, \text{if $e$ points into $v$,} \\
  - &\eps \int_{S_e} X \cdot n \, d \mathcal{H}^{n-1}, \,\, \text{if $e$ points out of $v$}.
\end{cases}
\end{equation}
We define the current similarly also at edges incident with  $u$ and $w$. If we label the edges as $e_1, \dots, e_l$ we have  a vector  ${\bm j}\in \R^{|E|}$ which has components ${\bm j}_k = j(e_k)$. By construction and by~\cref{eq:div-thm} ${\bm j}$ satisfies the so-called  Kirchhoff's current law, which means that at every vertex the current flowing in equals the current flowing out. We may write this simply as (see~\cite{Wagner})
\begin{align*}
  D {\bm j} = {\bm \delta}_1 -  {\bm \delta}_m
\end{align*}
where we have  labeled the vertices as $v_1, \dots, v_m$ with $v_1 = u$ and $v_m = w$,  and ${\bm \delta}_1$ and $ {\bm \delta}_m$ are as in \cref{bennyLemma}.

\subsection{Technical lemmas}

Before we prove the main theorem we recall the following  lemma from~\cite{BeLaVe}.
\begin{lemma}\label{l:hab:levelset:rough}
  Let $F$ be admissible. Let $x_u,x_w$ be as in \cref{def:islands} and assume that communication height from \cref{def:islands} is zero, i.e., $F(x_u;x_w) = 0$.
  If $\hat U_v$ is a component of $U_{-\delta/2} = \{ F < -\delta/2\}$, then 
  \begin{align*}
    \underset{\hat U_v}{\osc\,} h_{B_\varepsilon(x_u),B_\varepsilon(x_w)} \leq C  e^{-\frac{3\delta}{16 \eps}} ,
  \end{align*}
  for small enough $\eps \leq \eps_0$.
\end{lemma}
\begin{proof}
  The proof is almost the same as  in~\cite[Lemma 3.5]{BeLaVe}, but we repeat it for the reader's convenience. Let us denote $u := h_{B_\varepsilon(x_u),B_\varepsilon(x_w)}$ for short. 

  Recall that $\hat U_v$ is a component of $U_{-\delta/2} = \{ F < -\delta/2\}$.  Since $F$ is Lipschitz continuous, we find a Lipschitz domain $D_v$ such that 
  \begin{equation*}
    \hat U_v \subset D_v \subset  U_{-\frac{4\delta}{9}} = \{ F < -\tfrac{4\delta}{9}\}
  \end{equation*}
  and the Poincar\'e inequality holds in  $D_v$ with a constant that depends on $\|F\|_{C^{0,1}}$, i.e., 
  \begin{equation*}
    \int_{D_v} |u - u_{D_v}|^2 \, dx \leq C \int_{D_v} |\nabla u|^2 \, dx,
  \end{equation*}
  where $u_{D_v}$ denotes the average of $u$ in $D_v$.  We use the rough capacity bound \cref{eq:rough_cap} and $D_v \subset  U_{-4\delta/9}$ to deduce
  \begin{align*}
    \int_{D_v} |\nabla u|^2 \, dx &\leq e^{-\frac{4\delta}{9 \eps}}  \int_{D_v} |\nabla u|^2 e^{-\frac{F}{\eps}} \, dx \\
    &\leq \eps^{-1} e^{-\frac{4\delta}{9 \eps}}   \capacity(B_\varepsilon(x_u),B_\varepsilon(x_w))  \leq C \eps^{q-1} e^{-\frac{4\delta}{9 \eps}}.
  \end{align*}
  Fix a point $x_0 \in \hat U_v$. Then by Harnack's inequality~\cite[Lemma 2.7]{BeLaVe}  it holds
  \begin{equation*}
    \sup_{B_{\eps}(x_0)} |u - u_{D_v}| \leq \left( \fint_{2B_{\eps}(x_0)} |u - u_{D_v}|^2 \, dx\right)^{1/2} \leq C \eps^{-\frac{n}{2}}  \left(\int_{D_i} |u - u_{D_v}|^2 \, dx\right)^{1/2}.
  \end{equation*}
  In conclusion, we have (since $\eps \leq \eps_0$)
  \begin{equation*}
    \sup_{B_{\eps}(x_0)} |u - u_{D_v}|  \leq  \eps^{\frac{q-1-n}{2}} e^{-\frac{2\delta}{9 \eps}} \leq C e^{-\frac{3\delta}{16 \eps}} .
  \end{equation*}
  The claim follows from the fact that $x_0$ is arbitrary point in $\hat U_v$.
\end{proof}

We also need the following lemma which relates the function $\hat \eta$ to $\omega$ in the assumption \cref{eq:struc3}. This lemma can  be found in~\cite[Lemma 3.9]{BeLaVe}. 

\begin{lemma}\label{lem:convex}
  Assume that $G: \R^k \to \R$ is a convex function which has a proper minimum at the origin and let $\omega$ be  the increasing function from \cref{eq:struc3}. Then for a fixed  $\delta\leq \delta_0$ and for any $\eps \leq \eps_0$ it holds 
  \begin{equation*}
    (1 -\hat \eta(\eps))  \int_{\R^k} e^{-\frac{G(x)}{\eps}} \, dx\leq  \int_{\{ G < \delta\}} e^{-\frac{G(x)}{\eps}}  e^{\pm \frac{\omega(G(x))}{\eps}} \, dx \leq   (1 +\hat \eta(\eps))  \int_{\R^k} e^{-\frac{G(x)}{\eps}} \, dx,  
  \end{equation*}
for a continuous an increasing function $\hat \eta$ with $\hat \eta(0) = 0$, which depends on $\omega$ and on the dimension.
\end{lemma}

\subsection{Proof of the main theorem}\label{sec:mainthm}
We prove the main theorem by providing sharp lower bounds for the variation definition of the capacity and for \cref{def:thompson}, which is in some sense the dual of the argument in~\cite{LMS}.

\begin{proof}[\textbf{Proof of \cref{mainthm}}]
  Consider two local minima $x_u,x_w$, let $A = B_\eps(x_u)$ and $B = B_\eps(x_w)$, and let $h_{A,B}$ be the capacitary potential for the capacitor $(A,B)$. By rescaling we may assume that communication height from \cref{def:islands} is zero, i.e., $F(x_u;x_w)= 0$.

  \paragraph{\bf Lower bound:}
  Let $(G,\bm y)$ be the electrical network from \cref{ss:construction}, and label the vertices as $V = \{v_1,\ldots,v_m\}$, where $v_1 = u$, $v_m = w$.
  We need to show that
  \begin{equation*}
    \capacity(A,B) \geq (1 - C \hat \eta(\eps)) \frac{T(G;{\bm y})}{T(G/uw;{\bm y})},
  \end{equation*}
  where $\hat \eta$ is as in \cref{ss:main}.

  Let $\varphi:V \to \R$ be a function such that $\varphi(v) = h_{A,B}(x_v)$ where $v \in V$ and $x_v$ is the associated minimum point. Let $\hat U_v$ be the component of $\{ F < -\delta/2\}$ which contains $x_v$. By \cref{l:hab:levelset:rough} we have
  \begin{align*}
    \text{osc}_{\hat U_v}(h_{A,B}) \leq C  e^{-\frac{3\delta}{16 \eps}} \quad \text{for all $v \in V$}.
  \end{align*}
  Therefore, $\varphi$ satisfies
  \begin{align}\label{eq:vertex:approx}
    | h_{A,B}  - \varphi(v) | \leq   C  e^{-\frac{3\delta}{16 \eps}} \quad  \text{in } \,\hat U_{v} \quad \text{for $v \in V$.}
  \end{align}

  Consider an edge $e \in E$, which is not a loop, and let $v_{-},v_{+}$ be the two incident vertices in $V$. Denote the associated minimum points as $x_{-},x_{+}$, the associated islands as $U_{-},U_{+}$ respectively and the saddle point as $z_{e}$.
  We may assume that $z_e = 0$  and that the bridge is given by
  \begin{align*}
    O_{e} = O_{z_e,  \delta} = \{y_1:  g(y_1) <  \delta \} \times \{y':  G(y') <  \delta \}.
  \end{align*}
  Let us consider a domain (see \cref{f:localization2})
  \begin{equation*}
  \hat O_{e}  := \{y_1:  g(y_1) <  \delta \} \times \{y':  G(y') <  \delta/100 \}
  \end{equation*}
  and denote for $\tau \geq 0$ the surface 
  \begin{equation*}
    S_\tau :=  \{\tau\} \times \{y' :  G(y') < \delta/100\}.
  \end{equation*}

  \begin{figure}
    \begin{center}
      \begin{tikzpicture}
        \draw (-2,-2) .. controls (0,-1) and (0,1) .. (-2,2);
        \draw (-2,1) node {$\hat U_{v_-}$};
        \draw (2,-2) .. controls (0,-1) and (0,1) .. (2,2);
        \draw (2,1) node {$\hat U_{v_+}$};

        \filldraw[color=gray!60, fill=gray!30] (-2,0) circle (0.5);
        \draw (-2,0) node {$v_-$};

        \filldraw[color=gray!60, fill=gray!30] (2,0) circle (0.5);
        \draw (2,0) node {$v_+$};

        \draw[pattern=north west lines, pattern color=blue] (-0.8,-1.3) rectangle (0.8,1.3);
        \draw[pattern=north east lines, pattern color=blue] (-0.8,-0.4) rectangle (0.8,0.4);
        \draw (0,-1.7) node {$O_e$};

        \draw[pattern=north west lines, pattern color=blue] (-4,2) rectangle (-3.7,2.3);
        \draw (-4+0.6,2+0.1) node {$O_e$};

        \draw[pattern=north west lines, pattern color=blue] (-4,2-0.6) rectangle (-3.7,2.3-0.6);
        \draw[pattern=north east lines, pattern color=blue] (-4,2-0.6) rectangle (-3.7,2.3-0.6);
        \draw (-4+0.6,2+0.1-0.6) node {$\hat O_e$};

        \draw[pattern=north east lines, pattern color=blue] (-0.8,-0.4) rectangle (0.8,0.4);
      \end{tikzpicture}
    \end{center}
    \caption{The bridge $O_e$ connects the sets $\hat U_{v_+}$ and $\hat U_{v_-}$. The smaller cylindrical bridge $\hat O_e$ has its lateral boundaries inside $\hat U_{v_+} \cup \hat U_{v_-}$.}\label{f:localization2}
  \end{figure}
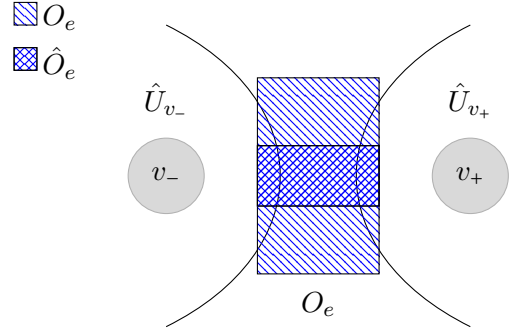

  We denote the lateral boundary of $ \hat O_{e} $  by $\Gamma_e :=  \{y_1:  g(y_1) =  \delta \} \times \{y':  G(y') <  \delta/100 \}$ and note that $\Gamma_e = S_{\tau_1}\cup S_{\tau_2}$ for $\tau_1<0 < \tau_2$ which satisfy $g(\tau_1) = g(\tau_2)= \delta$. Recall that we assume $F(z_e) = F(0)< \delta/3$ and therefore  by the definition of $\hat O_{e}$ and assumption \cref{eq:struc3} it holds  for all $y \in \Gamma_e$ that 
  \begin{equation} \label{eq:sides-estimate}
    \begin{split}
      F(y) \leq \overbrace{F(0)}^{<\delta/3} -  \overbrace{g(y_1)}^{=\delta} + \overbrace{G(y')}^{<\delta/100} + \overbrace{\omega(g(y_1))}^{<\delta/100} + \overbrace{\omega(G(y'))}^{<\delta/100}   < - \frac{\delta}{2}.
    \end{split}
  \end{equation}
  In other words, the lateral boundary $\Gamma_e$ is contained in the sublevel-set $\{ F < -\delta/2\} $ and the inequality in \cref{eq:vertex:approx} holds there.  

  Let us next prove that it holds
  \begin{align}\label{eq:gradient:bound}
    (1 - C \hat \eta(\eps)) (\varphi(v_{-})-\varphi(v_{+}))^2 {\bm y}_e \leq \eps\int_{O_{e}} |\nabla h_{A,B}|^2 e^{-\frac{F(y)}{\eps}} \,dy + C  e^{-\frac{\delta}{24 \eps}},
  \end{align}
  where the admittance $ {\bm y}_e$ is defined in~\cref{eq:geometric}. To this aim we fix $y' \in \{y':  G(y') <  \delta/100 \}$, let $\tau_1<0 < \tau_2$ be such that $g(\tau_1) = g(\tau_2)= \delta$ and notice that $( \tau_i, y') \in \Gamma_e$, for $i =1,2$. Using the fundamental theorem of calculus and \cref{eq:vertex:approx} we get
  \begin{align*}
    |\varphi(v_-) - \varphi(v_+)|-C  e^{-\frac{3\delta}{16 \eps}} &\leq |h_{A,B}(\tau_2, y') - h_{A,B}(\tau_1, y')| \\
    &\leq \int_{\{g < \delta\}} |\partial_{y_1} h_{A,B}(y_1,y')| dy_1 
    \\
    &= \int_{\{g < \delta\}} |\nabla h_{A,B}(y_1,y')| e^{-\frac{F(y)}{2\eps}} e^{\frac{F(y)}{2\eps}}dy_1.
  \end{align*}
  By Cauchy-Schwarz inequality we have
  \begin{align*}
    \begin{split}
      (\varphi(v_{-})-\varphi(v_{+}))^2- 2C e^{-\frac{3\delta}{8 \eps}} \leq \left( \int_{\{g<\delta\} } |\nabla h_{A,B}(y)|^2 e^{-\frac{F(y)}{\eps}} \,dy_1 \right) \left(  \int_{\{g<\delta\} }  e^{\frac{F(y)}{\eps}} \,dy_1\right)
    \end{split}
  \end{align*}
  for $ (y_1,y') \in \{ g <\delta\} \times \{ G <\delta/100\}$.  The assumption \cref{eq:struc3} implies
  \begin{equation*}
    F(y) \leq F(0) - g(y_1)  + \omega(g(y_1)) + G(y') + \omega(G(y')).
  \end{equation*}
  Dividing the above estimate by $e^{\frac{G(y')}{\eps}}e^{\frac{\omega(G(y'))}{\eps}} $ and integrating over $y'$  yields
  \begin{align*}
    \begin{split}
      &\left(\int_{\{G < \delta/100\}} e^{-\frac{G(y')}{\eps}}e^{-\frac{\omega(G(y'))}{\eps}}   dy' \right)
      \big((\varphi(v_{-})-\varphi(v_{+}))^2- 2C e^{-\frac{3\delta}{8 \eps}} \big) \\
&\leq\left( \int_{ \hat O_{e} } |\nabla h_{A,B}(y)|^2 e^{-\frac{F(y)}{\eps}} \,dy \right)\left(\int_{\{g<\delta\} }  e^{-\frac{g(y_1)}{\eps}} e^{\frac{\omega(g(y_1))}{\eps}} \,dy_1\right)e^{\frac{F(0)}{\eps}}.
    \end{split}
  \end{align*}
  Using \cref{lem:convex,lem:geometric-d} it holds
  \begin{equation*}
    \begin{split}
      e^{\frac{F(0)}{\eps}} \, \int_{\{g<\delta\} }  e^{\frac{-g(y_1)}{\eps}} e^{\frac{\omega(g(y_1))}{\eps}} \,dy_1 
      &\leq 
      (1 + \hat \eta(\eps)) e^{\frac{F(0)}{\eps}}  \int_{\R }  e^{\frac{-g(y_1)}{\eps}} \,dy_1 \\
      &\leq  
      (1 + C \hat \eta(\eps)) \, d_\eps(B_\eps(x_{-}),B_\eps(x_{+});\Omega_e),
    \end{split}
  \end{equation*}
  and, trivially
  \begin{align*}
    \int_{\{g < \delta\}} e^{\frac{-g(y_1)}{\eps}} e^{\frac{\omega(g(y_1))}{\eps}} dy_1 
    \geq 
    \int_{\{g < \eps\}} e^{\frac{-\eps}{2 \eps}} dy_1 
    \geq
    c |\{g < \eps\}|.
  \end{align*}
  Since $g$ is Lipschitz and $g(0)= 0$ we have  $(-c\eps,c\eps) \subset \{g < \eps\}$ for some $c$, and therefore $|\{g < \eps\}| \geq c \eps$.
  Again, by \cref{lem:convex,lem:geometric-d} we get
  \begin{equation*}
    \begin{split}
      e^{-\frac{F(0)}{\eps}} \, \int_{\{G < \delta/100\}} e^{-\frac{G(y')}{\eps}}e^{-\frac{\omega(G(y'))}{\eps}}   dy' 
      &\geq 
      (1 - \hat \eta(\eps))e^{-\frac{F(0)}{\eps}} \int_{\R^{n-1} }  e^{\frac{-G(y')}{\eps}} \,dy' \\
      &\geq  
      (1- C\hat \eta(\eps)) \, V_\eps(B_\eps(x_{-}),B_\eps(x_{+});\Omega_e),
      \end{split}
  \end{equation*}
  and trivially we also get
  \begin{align*}
    \int_{\{G < \delta/100\}} e^{\frac{-G(y')}{\eps}} e^{\frac{- \omega(G(y'))}{\eps}} dy' \leq |\{G < \delta\}|.
  \end{align*}
  Recalling that $F(0) \leq \delta/3$, this together with the above estimates and the definition of the admittance ${\bm y}_e$ \eqref{eq:geometric} imply the inequality~\cref{eq:gradient:bound}.

  Since \cref{eq:gradient:bound} holds for all $e \in E$ we can sum the inequalities over $e$ and rephrase the sum using the signed incidence matrix $D$ and the admittance matrix $Y$. To this aim, let $\bm \varphi$ be the vector $(\varphi(v_1),\ldots,\varphi(v_m))$, where $v_1 = u$ and $v_m = w$, and for an edge $e \in E$, let $v_{e^-},v_{e^+}$ be the incident vertices. Then since $D$ is the $|V |\times |E|$ signed incidence matrix, we have  for the edges $(e_1,\ldots,e_k)$
  \begin{align*}
    D^T {\bm \varphi} 
    = 
    (\varphi(v_{e^+_1})-\varphi(v_{e^-_1}),\ldots,\varphi(v_{e^+_k})-\varphi(v_{e^-_k})).
  \end{align*}
  Furthermore, by the definition of the admittance matrix $Y$ we have that
  \begin{align*}
    Y D^T {\bm \varphi} 
    = 
    ((\varphi(v_{e^+_1})-\varphi(v_{e^-_1}))y_{e_1},\ldots,(\varphi(v_{e^+_k})-\varphi(v_{e^-_k}))y_{e_k}).
  \end{align*}
  Recalling that \cref{eq:gradient:bound} holds for every edge $e \in E$, we get since sets $O_{e_i}$ are disjoint, that
  \begin{align*}
    (1 - C \hat \eta(\eps)) \langle DYD^T {\bm \varphi}, {\bm \varphi} \rangle \leq & 
    \eps\int_{\R^n} |\nabla h_{A,B}|^2 e^{-\frac{F(y)}{\eps}} \,dy + C e^{-\frac{\delta}{24 \eps}} 
    \\
    \leq &
    \capacity(A,B) + C e^{-\frac{\delta}{24 \eps}}.
  \end{align*}
  Now note that, the rough capacity bound \cref{eq:rough_cap} implies that $\capacity(A,B) \geq c_1 \eps^{q_1}$. By construction, it holds ${\bm \varphi}_1 = \varphi(u) = 1$ and  ${\bm \varphi}_m = \varphi(w) = 0$, therefore \cref{bennyLemma} completes the proof of the lower bound.

  \paragraph{\bf Upper bound:}

  We prove the upper bound by a similar argument by providing a lower bound in the dual characterization~\cref{def:thompson}. Indeed, by the second inequality in~\cref{eq:capa-local} this provides an upper bound for the global capacity.  Let us fix a vector field $X \in \mathcal{M}$, where $\mathcal{M}$ is defined via conditions~\cref{def:def-M}, and construct the associated current  ${\bm j} \in \R^{|E|}$ as in  \cref{sebsec:construction}. The construction implies that ${\bm j} $  satisfies Kirchhoff's current law  $D {\bm j} = {\bm \delta}_1 -  {\bm \delta}_m$, and therefore it holds by  \cref{bennyLemma} and \cref{lem:dual-net} that
  \begin{align*}
    \langle Y^{-1} {\bm j} , {\bm j} \rangle \geq \frac{T(G/uw;{\bm y})}{T(G;{\bm y})}.
  \end{align*}

  In order to conclude the proof, it is enough to show that at every edge $e \in E$ it holds
  \begin{equation}\label{eq:mainthm-upper}
    \eps \int_{O_e \cap U_{\delta/3}} |X|^2 e^{\frac{F}{\eps}}\, dx \geq (1 -  C \hat \eta(\eps)) \frac{{\bm j}_e^2}{{\bm y}_e},
  \end{equation}
  where $O_{e} = O_{z_e, \delta}$ denotes the associated bridge. To this aim we may choose the coordinates in $\R^n$ such that
  \begin{align*}
    O_e = \{x_1 :  g(x_1) < \delta\} \times \{x' :  G(x') < \delta\}.
  \end{align*}
  For every $|\tau|< \delta/100$  redefine 
  \begin{equation}\label{eq:S-tau}
    S_\tau := \{\tau\} \times \{x' :  G(x') \leq \delta\},
  \end{equation}
  and note that by the definition of ${\bm j}$ in~\cref{def:current}  it holds
  \begin{equation}\label{eq:from-current}
    \eps  \, \left |\int_{S_{0}} X \cdot \hat e_1 \,  d \mathcal{H}^{n-1} \right|  = |{\bm j}_e|,
  \end{equation}
  where $\hat e_1$ is the first coordinate vector of $\R^n$.  Let us fix $0< \tau < \delta/100$  and consider the  domain (see \cref{f:localization3})
  \begin{align*}
    \hat O_\tau = \{x_1: 0 < g(x_1) < \tau \} \times \{x' :  G(x') < \delta\}
  \end{align*}
  and denote the `cylindrical' boundary by 
  \begin{align*}
    \Sigma_\tau =  \{x_1 :  0 \leq g(x_1) \leq \tau \} \times \{x' :  G(x') = \delta\}.
  \end{align*}
  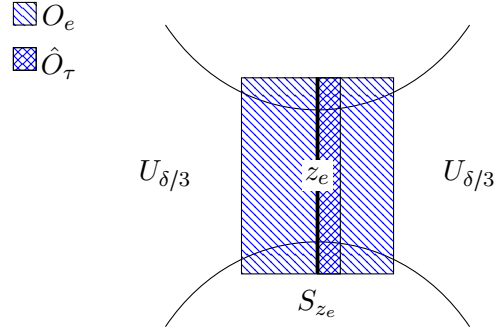
\begin{figure}
    \begin{center}
      \begin{tikzpicture}
        \draw (-2,-2) .. controls (-1,-0.5) and (1,-0.5) .. (2,-2);

        \draw (-2,2) .. controls (-1,0.5) and (1,0.5) .. (2,2);
  
        \draw (-2,0) node {$U_{\delta/3}$};
        \draw (2,0) node {$U_{\delta/3}$};
  
        \draw[pattern=north west lines, pattern color=blue] (-1,-1.3) rectangle (1,1.3);
        \draw[pattern=north east lines, pattern color=blue] (0,-1.3) rectangle (0.3,1.3);
        \draw[line width=0.5mm] (0,-1.3) to (0,1.3);
        \draw (0,-1.7) node {$S_{z_e}$};
  
        \draw[pattern=north west lines, pattern color=blue] (-4,2) rectangle (-3.7,2.3);
        \draw (-4+0.6,2+0.1) node {$O_e$};
  
        \draw[pattern=north west lines, pattern color=blue] (-4,2-0.6) rectangle (-3.7,2.3-0.6);
        \draw[pattern=north east lines, pattern color=blue] (-4,2-0.6) rectangle (-3.7,2.3-0.6);
        \draw (-4+0.6,2+0.1-0.6) node {$\hat O_\tau$};

        \fill[white] (-0.2,-0.2) rectangle (0.2,0.2);
        \draw (0,0) node {$z_e$};
      \end{tikzpicture}
    \end{center}
    \caption{The set $S_{z_e}$ and the domain $\hat O_\tau$ separate the domain $U_{\delta/3}$ into different components.}\label{f:localization3}
  \end{figure}
  Arguing as in~\cref{eq:sides-estimate} we deduce that $F > \delta/3$ on $\Sigma_\tau$ and therefore  $\Sigma_\tau \subset (\overline U_{\delta/3})^c$. Note that the `lateral' boundary of $ \hat O_\tau$ is the union of $S_0$ and $S_\tau$ defined in \cref{eq:S-tau}.   By~\cref{def:def-M} $X$ is divergence free, and thus we obtain by the divergence theorem that
  \begin{align*}
    0 
    =& 
    \int_{\hat O_e \cap  U_{\delta/3}}  \text{div} (X)  \, dx 
    \\
    =& 
    \int_{\partial U_{\delta/3} \cap \hat O_e}  X \cdot n   \, d \mathcal{H}^{n-1} + \int_{S_0}  X \cdot n   \, d \mathcal{H}^{n-1} + \int_{S_\tau}  X \cdot n   \, d \mathcal{H}^{n-1} .
  \end{align*}
  Again by~\cref{def:def-M}  we have $X \cdot n = 0$ on $\partial U_{\delta/3}$,  and since the normal on  the lateral boundary,  $S_0$ and $S_\tau$, points in direction of $\hat e_1$, we have  by \cref{eq:from-current} 
  \begin{align*}
    \eps \, \left |\int_{S_{\tau}} X \cdot \hat e_1 \,  d \mathcal{H}^{n-1} \right | = \eps \, \left |\int_{S_{0}} X \cdot \hat e_1 \,  d \mathcal{H}^{n-1} \right| = |{\bm j}_e|.
  \end{align*}
  We may apply the same argument to $\tau <0$ to deduce the above equality for all $|\tau|< \delta/100$.  

  We proceed by the Cauchy-Schwarz inequality 
  \begin{align*}
    |{\bm j}_e| = \eps \, \left |\int_{S_{\tau}} X \cdot \hat e_1 \,  d \mathcal{H}^{n-1} \right | \leq  \eps \, \left(  \int_{S_{\tau}} |X|^2 e^{\frac{F}{\eps}} \,  d \mathcal{H}^{n-1} \right)^{\frac12} \left( \int_{S_{\tau}}  e^{-\frac{F}{\eps}}\,  d \mathcal{H}^{n-1}\right)^{\frac12}.
  \end{align*}
  By assumption \cref{eq:struc3} we have
  \begin{align*}
    F(y) \geq F(0) - g(y_1) - \omega(g(y_1)) + G(y') - \omega(G(y'))
  \end{align*}
  thus by \cref{lem:convex,lem:geometric-d} it holds
  \begin{equation*}
    \begin{split}
      \int_{S_{\tau}}  e^{-\frac{F}{\eps}}\,  d \mathcal{H}^{n-1} &\leq  e^{\frac{g(\tau)}{\eps} } e^{\frac{\omega(g(\tau))}{\eps} } e^{-\frac{F(0)}{\eps}} \int_{\{ G <\delta\}} e^{-\frac{G(x')}{\eps} } e^{\frac{\omega(G(x'))}{\eps} }  \, dx' \\
      &\leq (1+   \hat \eta(\eps)) e^{\frac{g(\tau)}{\eps} } e^{\frac{\omega(g(\tau))}{\eps} } e^{-\frac{F(0)}{\eps}} \int_{\R^{n-1}} e^{-\frac{G(x')}{\eps} }  \, dx' \\
      &\leq (1+  C \hat \eta(\eps))e^{\frac{g(\tau)}{\eps} } e^{\frac{\omega(g(\tau))}{\eps} }V_\eps(B_\eps(x_{-}),B_\eps(x_{+});\Omega_e).
    \end{split}
  \end{equation*}
  Hence, by the three previous inequalities we have
  \begin{align*}
    \frac{{\bm j}_e^2}{V_\eps(B_\eps(x_{-}),B_\eps(x_{+});\Omega_e)}  e^{-\frac{g(\tau)}{\eps} } e^{-\frac{\omega(g(\tau))}{\eps} }\leq  (1+ C\hat \eta(\eps)) \eps^2 \int_{S_{\tau}} |X|^2 e^{\frac{F}{\eps}} \,  d \mathcal{H}^{n-1}.
  \end{align*}
  Integrating over $\tau \in (-\delta/100,\delta/100)$ and using \cref{lem:convex} and  \cref{lem:geometric-d} we get
  \begin{align*}
    {\bm j}_e^2 e^{-\frac{F(0)}{\eps}} \frac{d_\eps(B_\eps(x_{-}),B_\eps(x_{+});\Omega_e)}{V_\eps(B_\eps(x_{-}),B_\eps(x_{+});\Omega_e)} \leq (1+  C \hat \eta(\eps)) \eps^2  \int_{O_e} |X|^2 e^{\frac{F}{\eps}}\, dx.
  \end{align*}
  Inequality~\cref{eq:mainthm-upper} then follows from the definition of ${\bm y}_e$ in~\cref{eq:geometric}.
\end{proof}

\section*{Acknowledgments}
B.A. was supported by the Swedish Research Council dnr: 2019--04098.
V.J. was  supported by the Academy of Finland grant 314227.

\end{document}